%% file: Criterion_Quillen_functors__arxiv_.tex
\documentclass[11pt,a4paper]{article}
\pdfoutput=1

\usepackage[margin=1in]{geometry}

\input{Auxiliary/PackagesAndOperators}

 \newtheorem{thm}{Theorem}[section]
  \newtheorem{taggedtheoremx}{Theorem}
 \newenvironment{taggedtheorem}[1]
 {\renewcommand\thetaggedtheoremx{#1}\taggedtheoremx}
 {\endtaggedtheoremx}
 \newtheorem{cor}[thm]{Corollary}
 \newtheorem{lem}[thm]{Lemma}
 \newtheorem{prop}[thm]{Proposition}

 \theoremstyle{definition}

 \theoremstyle{remark}
 \newtheorem{rem}[thm]{Remark}

\title{
	\begin{Large}
		\textsc{A simple characterization of Quillen adjunctions}
	\end{Large}
}
\author{Victor Carmona\thanks{The author was supported by the Max-Planck Institute for Mathematics in the Sciences in Leipzig and partially supported by the grant PID2020-117971GB-C21 funded by MCIN/AEI/10.13039/501100011033.\\
		$\,$\\
		\textit{Affiliation:} Max Planck Institut f\"ur Mathematik in den Naturwissenschaften, Inselstra\ss e 22, 04103 Leipzig, Germany.\\
		\textit{Email address:} \texttt{victor.carmona@mis.mpg.de}}
}

\date{\today}

\begin{document}

\maketitle
 
 \begin{abstract}
We observe that an enriched right adjoint functor between model categories which preserves acyclic fibrations and fibrant objects is quite generically a right Quillen functor.  
 \end{abstract}

%\tableofcontents

\section{Introduction}
	\input{Contents/Intro}

\section{The criterion for model categories}\label{sect_Qpair between model}
	\input{Contents/ModelCats}

\section{The criterion for semimodel categories}\label{sect_Qpair between semimodel}
\input{Contents/SemiModelCats}

\section*{Acknowledgments}
The author would like to thank David White for 
pointing out Theorem \ref{thm_Lurie} to the author and for his always kind correspondence.

\bibliographystyle{acm}
%\bibliography{Bibliography}

\end{document}

%% file: Auxiliary/PackagesAndOperators.tex
\usepackage[english,activeacute]{babel}
\usepackage{etex}
\usepackage{latexsym}
\usepackage{mathtools}
\usepackage{amsmath}
\usepackage{amssymb}
\usepackage{wasysym}
\usepackage{mathrsfs}
\usepackage{euscript}
\usepackage{amsthm}
\usepackage{amsfonts}
\usepackage{cancel}
\usepackage{multicol}
\usepackage{color}
\usepackage{multirow}
\usepackage{makeidx}
\usepackage{upgreek}
\usepackage{setspace} 
\usepackage{import}
\usepackage{float}
\usepackage{mleftright}
\usepackage{stmaryrd}
\usepackage{diagbox}

\usepackage{mathalpha}

\DeclareFontFamily{U}{txcal}{\skewchar \font =45}
\DeclareFontShape{U}{txcal}{m}{n}{<-> txr-cal}{}
\DeclareMathAlphabet{\mathcalpxtx}{U}{txcal}{m}{n}

\DeclareMathAlphabet{\mathdutchcal}{U}{dutchcal}{m}{n}

\DeclareFontFamily{U}{stixtwobb}{\skewchar\font=45}%
\DeclareFontShape{U}{stixtwobb}{m}{n}{<->\mathalfa@bbscaled stix2-mathbb}{}
\DeclareMathAlphabet{\mathbbst}{U}{stixtwobb}{m}{n}

\usepackage{bbold}

\usepackage{tikz}
\usepackage{tikz-cd}
\usetikzlibrary{matrix,arrows,calc,decorations.pathmorphing,decorations.pathreplacing,positioning}

\usepackage{xcolor}
\definecolor{babyblueeyes}{rgb}{0.54, 0.81, 0.94}
\definecolor{blizzardblue}{rgb}{0.67, 0.9, 0.93}
\definecolor{blue(munsell)}{rgb}{0.0, 0.5, 0.69}
\definecolor{bluegray}{rgb}{0.4, 0.6, 0.8}
\definecolor{bondiblue}{rgb}{0.0, 0.58, 0.71}
\definecolor{cadetblue}{rgb}{0.37, 0.62, 0.63}
\definecolor{carolinablue}{rgb}{0.6, 0.73, 0.89}
\definecolor{cinnamon}{rgb}{0.82, 0.41, 0.12}
\definecolor{darkcandyapplered}{rgb}{0.64, 0.0, 0.0}
\definecolor{darkcyan}{rgb}{0.0, 0.55, 0.55}
\definecolor{darkmidnightblue}{rgb}{0.0, 0.2, 0.4}
\definecolor{darkpastelblue}{rgb}{0.47, 0.62, 0.8}
\definecolor{frenchblue}{rgb}{0.0, 0.45, 0.73}

\usepackage{url}
\usepackage[bookmarksopen,colorlinks,allcolors=frenchblue,urlcolor=frenchblue]{hyperref}

\makeindex

\usepackage{scalerel}

\usepackage{wrapfig}
\usepackage{ifpdf}

\usepackage{syntonly}
\DeclareMathOperator{\Lrect}{\mathsf{L}}

\DeclareMathOperator{\AFib}{\mathsf{AFib}}
\DeclareMathOperator{\Cof}{\mathsf{Cof}}
\DeclareMathOperator{\Fib}{\mathsf{Fib}}

\DeclareMathOperator{\Mcal}{\mathcalpxtx{M}}
\DeclareMathOperator{\Ncal}{\mathcalpxtx{N}}
\DeclareMathOperator{\Vcal}{\mathcalpxtx{V}}

\DeclareMathOperator{\Ho}{\mathsf{Ho}}

\DeclareMathOperator{\Hom}{\mathsf{Hom}}

\usepackage{titlesec}

\usepackage{float}

\usepackage{fancyhdr}
\usepackage{tikz}
\usetikzlibrary{calc,decorations.pathmorphing,shapes}

%% file: Contents/Intro.tex
In this note, we provide one of the easiest characterizations of Quillen adjunctions that may exist (Theorem \ref{thm:Criterion for model cats}). To the best of our knowledge, the only previous characterizations of Quillen pairs that have appeared in the literature are:
\begin{taggedtheorem}{\textbf{J}}\cite[Proposition 6.18]{joyal_quasicat_2007} \label{thm_Joyal} An adjunction $F\colon \Mcal\rightleftarrows \Ncal:\! U$ between model categories is a Quillen pair iff one of the following equivalent conditions holds:
	\begin{itemize}
		\item $F$ preserves acyclic cofibrations and cofibrations with cofibrant domain.
		\item $U$ preserves acyclic fibrations and fibrations with fibrant codomain.
	\end{itemize}
\end{taggedtheorem}

\begin{taggedtheorem}{\textbf{L}}\cite[Corollary A.3.7.2]{lurie_htt_2009} \label{thm_Lurie} 
If $\Mcal$ and $\Ncal$ are simplicial model categories and $\Ncal$ is a left proper model category, then a $\mathsf{sSet}$-enriched adjunction
$F\colon \Mcal\rightleftarrows\Ncal:\! U$
is a Quillen pair iff $F$ preserves cofibrations and $U$ preserves fibrant objects.
\end{taggedtheorem}
%It is clear that Theorem \ref{thm:Criterion for model cats} is similar to Theorem \ref{thm_Lurie}. %Interestingly, our proof is simpler than that of Theorem \ref{thm_Lurie}.

Additionally, we observe that our criterion holds in the setting of semimodel categories (Proposition \ref{prop:Criterion for leftmodel cats}). As a direct consequence of this, we obtain a criterion to construct right Quillen functors into/from a left Bousfield localization (Corollary \ref{cor: Qpairs out of semimodel LBL}). For right Quillen functors out of such a localization (in semimodel categories), this is the only known  result that might be applied due to Remark \ref{rem: No characterization of core fibrations in LBL}.

A similar version of this last result (Lemma \ref{lem:Benini-Schenkel}) greatly simplifies some arguments in an upcoming work with M.\ Benini and A.\ Schenkel about the strictification of the time-slice axiom for locally covariant AQFTs in any dimension. In \cite{benini_strictification_2023} you can learn about these words.

\paragraph*{Outline.} Section \ref{sect_Qpair between model} contains the announced characterization of Quillen pairs. Section \ref{sect_Qpair between semimodel} presents the semimodel version and the application mentioned above.

%% file: Contents/ModelCats.tex
Let us fix from now on a closed symmetric monoidal model category $\Vcal$ as the base for enrichment.

\begin{thm}\label{thm:Criterion for model cats} Let $F\colon \Mcal\rightleftarrows \Ncal:\! U $ be a $\Vcal$-adjunction between $\Vcal$-model categories. Assume that $\Mcal$ admits a generating set of acyclic cofibrations $J$ with cofibrant domains. Then, $F\dashv U$ is a Quillen pair iff $U$ preserves acyclic fibrations and fibrant objects.
\end{thm}
\begin{proof} It suffices to prove $(\Leftarrow)$, i.e.\ for any  fibration $f\in\Fib_{\Ncal}$, it holds $U(f)\in \Fib_{\Mcal}$. We have to show that $\forall j\in J$, $j\boxslash U(f)$ (lifting property). By adjunction, we are reduced to prove that $F(j)$ is an acyclic fibration $\forall j$. Since $U$ preserves $\AFib$, we know that $F(j)\in \Cof$ and also that it has cofibrant domain. It remains to check that $F(j)$ is a weak equivalence. For this, we use:
	\begin{itemize}
		\item The class of weak equivalences in $\Mcal$ is the preimage of the class of isomorphisms through the localization functor $\Mcal\to \Ho\Mcal$. Hence,
		\item A map $g\colon X\to Y$ is a weak equivalence iff for any fibrant object $Z\in \Ncal$, 
		$$
		g^*\colon [Y,Z]_{\Ncal}\longrightarrow [X,Z]_{\Ncal}
		$$
        is a bijection.	
	\end{itemize} 
Let $Z\in \Ncal$ be a fibrant object and $\mathbb{I}_c\in \Vcal$ be a cofibrant replacement of the monoidal unit. The claim follows from the following commutative diagram
$$
\begin{tikzcd}[ampersand replacement=\&]
\left[F(B), Z\right]_{\Ncal}\ar[d,"F(j)^*"'] \ar[r, "\simeq"]\& \left[\,\mathbb{I}_c\otimes F(B),Z\right]_{\Ncal} \ar[r, "\simeq"]\& \left[\,\mathbb{I}_c, 
\underline{\Hom}_{\Ncal}(F(B),Z)\right]_{\Vcal}\ar[r, "\simeq"] \& \left[B,U(Z)\right]_{\Mcal}\ar[d,"j^*"]\\
\left[F(A), Z\right]_{\Ncal} \ar[r, "\simeq"']\& \left[\,\mathbb{I}_c\otimes F(A),Z\right]_{\Ncal} \ar[r, "\simeq"']\& \left[\,\mathbb{I}_c, 
\underline{\Hom}_{\Ncal}(F(A),Z)\right]_{\Vcal}\ar[r, "\simeq"'] \& \left[A,U(Z)\right]_{\Mcal}
\end{tikzcd},
$$
where we have applied: $(i)$ $A$ and $B$ are cofibrant objects (and so all the expressions that appear, e.g.\ 
$$
\mathbb{I}_c\otimes A,\quad \quad \mathbb{I}_c\otimes F(B),\quad \quad \underline{\Hom}_{\Mcal}(A,U(Z))\quad \dots
$$
are already derived), $(ii)$ the axioms of $\Vcal$-model category, $(iii)$ $F\dashv U$ is a $\Vcal$-adjunction, and $(iv)$ $j\colon A\to B$ is a weak equivalence. 
\end{proof}

\begin{rem} It is tempting to think that the argument can be generalized dropping  enrichments and the hypothesis on $J$, even its existence. The argument can be pursued up to the point where one uses some sort of adjunction between $F$ and $U$ between localizations. See the proof of Lemma \ref{lem:Benini-Schenkel}.% The problem is that it is not ensured, at least the author does not know why, that $F$ preserves cylinders and $U$ preserves paths.
\end{rem}

%% file: Contents/SemiModelCats.tex
The power of our characterization of Quillen pairs really hits when working with semimodel categories. In fact, this note was motivated by the technical problem mentioned in Remark \ref{rem: No characterization of core fibrations in LBL}. References for the theory of semimodel categories are  \cite{barwick_left_2010, carmona_when_2023}.

As one can learn from \cite{carmona_when_2023}, semimodel structure may refer to a variety of notions. In this document, we will only consider Spitzweck semimodel categories (see \cite[Definition 2.5]{carmona_when_2023}) for the sake of brevity and conciseness. %Also, we will adopt the following notation:
%\begin{defn} A \emph{core cofibration} (resp.\ \emph{core fibration}) is a cofibration that has cofibrant source (resp.\ fibrations with fibrant target). We denote the class of core cofibrations (resp.\ core fibrations) by $\Cof_{\circ}$ (resp.\ $\Fib_{\circ}$). Also, $\ACof_\circ \equiv \Cof_{\circ}\cap \Eq$ and $\AFib_{\circ}\equiv \Fib_{\circ}\cap \Eq$.
%\end{defn}

The proof of Theorem \ref{thm:Criterion for model cats} also yields: 
\begin{prop}\label{prop:Criterion for leftmodel cats}
	 Let $F\colon \Mcal\rightleftarrows \Ncal:\! U $ be a $\Vcal$-adjunction between left $\Vcal$-semimodel categories, where $\Mcal$ has a generating set of acyclic cofibrations $J$ with cofibrant domains. Then, the following conditions are equivalent:
	\begin{itemize}
		\item[(1)] $U$ preserves (acyclic) fibrations, i.e.\ $F\dashv U$ is a Quillen pair.
		\item[(2)] $U$ preserves acyclic fibrations and fibrant objects.
	\end{itemize}
\end{prop}

\begin{rem} Recall that, in general, fibrations in a left semimodel category are not characterized by a lifting property (see \cite[Lemma 1.7 and Corollary 1.8]{barwick_left_2010}). 
\end{rem}

\paragraph*{Application: Quillen adjunctions out of Bousfield localizations}$\,$\newline
Existence results for left Bousfield localizations of semimodel categories can be found in \cite{white_left_2020,carmona_when_2023}.

For ease of exposition, we assume that the left semimodel categories in the following result are tractable, i.e.\ they are cofibrantly generated by sets of arrows with cofibrant domains.\footnote{Also because tractability is one necessary condition for \cite[Theorem A]{white_left_2020} and \cite[Theorem B]{carmona_when_2023}.}
\begin{cor}\label{cor: Qpairs out of semimodel LBL} Let $\Mcal$, $\Ncal$ and $\Ncal'$ be tractable left $\Vcal$-semimodel categories and $S$ be a set of maps in $\Mcal$. Assume that the $\Vcal$-enriched left Bousfield localization\footnote{ See \cite[Theorem 4.46]{barwick_left_2010} and the subsection that contains it for definitions about enriched Bousfield localization.} $\Lrect_S\Mcal$ exists as a tractable left $\Vcal$-semimodel category (e.g.\ $\Mcal$ is cellular or combinatorial). Consider that we are given right $\Vcal$-adjoint functors $U\colon \Ncal \to \Mcal$ and $U'\colon \Mcal\to \Ncal'$. Then,
 \begin{itemize}
			\item[$(a)$] $U$ induces a  Quillen pair $F\colon  \Lrect_{S}\Mcal\rightleftarrows \Ncal:\! U$ iff $U$ preserves $\Fib$ and  sends fibrant objects to fibrant $S$-local objects.
			\item[$(b)$] $U'$ induces a  Quillen pair $F'\colon \Ncal'\rightleftarrows \Lrect_{S}\Mcal:\! U'$ iff $U'$ preserves $\Fib$ and sends fibrant $S$-local objects to fibrant objects.
		\end{itemize}
\end{cor}
\begin{proof} Direct application of Proposition \ref{prop:Criterion for leftmodel cats}. Being fibrant in $\Lrect_S\Mcal$ and being fibrant $S$-local are equivalent conditions (see the last paragraph of \cite[Proof of Theorem 4.2]{white_left_2020} and \cite[Proposition 5.9]{carmona_when_2023}).
\end{proof}
\begin{rem}
$(b)$ is not only new, but necessary to find right Quillen functors with domain  $\Lrect_S\Mcal$. The problem is the lack of a characterization for core fibrations in $\Lrect_T\Ncal$. The best we have is a characterization of core fibrations with \textbf{co}fibrant domain (\cite[Remark 4.5]{white_left_2020}).
\end{rem}

A similar result, although not requiring $\Vcal$-enrichment, is:
\begin{lem}\label{lem:Benini-Schenkel} Let $F\colon\Mcal\rightleftarrows \Ncal:\! U$ be a Quillen pair between tractable left semimodel categories. Assume we are given sets of maps $S$ and $T$ of $\Mcal$ and $\Ncal$ respectively, such that $\Lrect_S\Mcal$ and $\Lrect_T\Ncal$ exist as tractable left semimodel categories. Then, $F\colon L_S\Mcal\rightleftarrows \Lrect_T\Ncal:\! U$ is a Quillen pair iff $U$ sends $T$-local objects to $S$-local objects.	
\end{lem}
\begin{proof} As in the proof of Theorem \ref{thm:Criterion for model cats}, we are reduced to check that $F(j)$ is a weak equivalence when $j\in J$, generating set of acyclic cofibrations in $L_S\Mcal$. Since $F\colon \Mcal \rightleftarrows \Ncal:\! U$ is already a Quillen adjunction, we have a commutative diagram
	$$
	\begin{tikzcd}[ampersand replacement=\&]
		\left[F(B), Z\right]_{\Lrect_T\Ncal}\ar[d,"F(j)^*"'] \ar[r, equal]\& \left[F(B), Z\right]_{\Ncal} \ar[r, "\simeq"]\& \left[B, U(Z)\right]_{\Mcal}\ar[r, equal] \& \left[B,U(Z)\right]_{\Lrect_S\Mcal}\ar[d,"j^*"]\\
		\left[F(A), Z\right]_{\Lrect_{T}\Ncal} \ar[r, equal]\& \left[F(A), Z\right]_{\Ncal} \ar[r, "\simeq"']\& \left[A, U(Z)\right]_{\Mcal}\ar[r, equal] \& \left[A,U(Z)\right]_{\Lrect_S\Mcal}
	\end{tikzcd}
	$$
for any fibrant $T$-local object $Z\in \Ncal$ (we use $(i)$ and $(iv)$ from the proof of Theorem \ref{thm:Criterion for model cats}). Note that we also used that $\Ho \Lrect_T\Ncal$ (resp.\ $\Ho \Lrect_T\Mcal$) is a full subcategory of $\Ho\Ncal$ (resp.\ $\Ho\Mcal$). We conclude the result as in Theorem \ref{thm:Criterion for model cats} because being fibrant in $\Lrect_S\Mcal$ and being fibrant $S$-local are equivalent conditions
\end{proof}

\begin{rem}\label{rem: No characterization of core fibrations in LBL} This lemma is equivalent to the proof of the universal property in \cite[Theorem 4.3]{white_left_2020}. For one implication, just take $T=\emptyset$, and for the other, apply the universal property to the composition 
	$
\Mcal\rightleftarrows\Ncal\rightleftarrows \Lrect_T\Ncal.	
	$
 The cited result is  the source of inspiration for Theorem \ref{thm:Criterion for model cats}. We have presented Lemma \ref{lem:Benini-Schenkel} in this form for future reference. 
\end{rem}